\documentclass[12pt]{amsart}
\usepackage{amssymb}
\usepackage[all]{xy}

\newtheorem{theorem}{Theorem}[section]
\newtheorem{definition}[theorem]{Definition}
\newtheorem{proposition}[theorem]{Proposition}

\begin{document}

\title[Complex half-classical geometry]{Complex analogues of the half-classical geometry}

\author{Teodor Banica}
\address{T.B.: Department of Mathematics, University of Cergy-Pontoise, F-95000 Cergy-Pontoise, France. {\tt teo.banica@gmail.com}}

\author{Julien Bichon}
\address{J.B.: Department of Mathematics, University of Clermont Auvergne, F-63178 Aubiere, France. {\tt julien.bichon@uca.fr}}

\subjclass[2010]{14A22 (16S38, 46L65)}
\keywords{Quantum isometry, Noncommutative manifold}

\begin{abstract}
Under very strong axioms, there is precisely one real noncommutative geometry between the classical one and the free one, namely the half-classical one, coming from the relations $abc=cba$. We discuss here the complex analogues of this geometry, notably with a study of the geometry coming from the commutation relations between all the variables $\{ab^*,a^*b\}$, that we believe to be the ``correct'' one.
\end{abstract}

\maketitle

\section*{Introduction}

The fact that the behavior of the subatomic particles might be described by some kind of ``noncommutative geometry'', with a great deal of probability theory involved, is as old as quantum mechanics. While the weak and strong forces are radically different from gravity and electromagnetism, one hope, however, would be that this noncommutative geometry could simply appear as an ``analogue'' of the classical geometry.

From this perspective, any exploration of the noncommutative analogues of the various aspects of the classical geometry can only be useful. The subject is, of course, still in its infancy. As a reminder here, the classical geometry itself took about 2000 years to be axiomatized, and applied to basic problems in physics (Kepler, Newton). Noncommutative geometry seems to be on a faster track, but there is no reason to be overly optimistic, and not to remain modest. After all, we are dealing here with phenomena that both our human brains and machinery have big big troubles in observing and understanding. 

Doing some abstract mathematics, with these ideas in mind, will be our purpose here. We will be interested in noncommutative algebraic geometry, of very elementary type: basic curves and surfaces. Technically speaking, we will use the same formalism and philosophy as Connes \cite{con}, a noncommutative space being for us the dual of an operator algebra. For some successful applications of this philosophy, we refer to \cite{cco}.

Our starting point is a recent discovery, from \cite{bsp}, \cite{bve}, stating that when imposing the strongest possible axioms, there are only three real geometries, namely the classical one, the free one, and an intermediate one, called ``half-classical''. Motivated by this fact, we started in \cite{bbi} an investigation of the possible half-classical complex geometries. We will finish here the work started in \cite{bbi}, by identifying the ``standard'' such geometry.

Let us first explain the above-mentioned rigidity phenomenon, from the real case. From an elementary, all-around point of view, the basic objects of the $N$-dimensional geometry are the unit sphere, the standard cube, and the orthogonal group:
\begin{eqnarray*}
S^{N-1}_\mathbb R&=&\left\{x\in\mathbb R^N\Big|\sum_ix_i^2=1\right\}\\
T_N&=&\left\{x\in\mathbb R^N\Big|x_i=\pm\frac{1}{\sqrt{N}}\right\}\\
O_N&=&\left\{U\in M_N(\mathbb  R)\Big|U^t=U^{-1}\right\}
\end{eqnarray*}

Note that we have not included $\mathbb R^N$ itself in our list. This is because we would like later on to talk about noncommutative versions of the above objects, and our formalism here requires all the spaces to be compact. Physically speaking, our belief is that for certain key problems, such as those regarding QCD, this restriction is not important.

Quite remarkably, there is a full set of connections between the above objects:
\begin{enumerate}
\item $O_N$ is the isometry group of $S^{N-1}_\mathbb R$, and $S^{N-1}_\mathbb R$ appears as $O_N(\circ)$, where $\circ=(1,0,\ldots,0)$. In addition, we have an embedding $O_N\subset\sqrt{N}\cdot S^{N^2-1}_\mathbb R$.

\item $T_N$ appears inside $S^{N-1}_\mathbb R$ by setting $|x_1|=\ldots=|x_N|$. Conversely, $S^{N-1}_\mathbb R$ appears from $T_N\subset\mathbb R^N$ by ``deleting'' this relation, while still keeping $\sum_ix_i^2=1$.

\item $T_N\simeq\mathbb Z_2^N$ is a maximal compact abelian subgroup of $O_N$, and the group $O_N$ itself can be reconstructed from this subgroup, by using various methods.
\end{enumerate}

We are of course a bit vague here, but it is not hard to believe that, with a minimal knowledge of basic algebraic geometry and representation theory, all the $2\times3=6$ correspondences can indeed be established. This is actually a very good exercise.

Let us discuss now the construction of the noncommutative versions of the above objects. There are several possible choices here, and based on our personal knowledge of quantum mechanics, and of mathematical physics in general, we will use the operator algebra formalism. The idea indeed is that whenever we have a unital $C^*$-algebra $A$, we can write $A=C(X)$, with $X$ being a noncommutative compact space. This is supported by a non-trivial theorem of Gelfand, which states that when $A$ is commutative, the formula $A=C(X)$ holds indeed, with $X$ being a classical space, called spectrum of $A$.

So, let us define the free sphere, free cube, and free orthogonal group, by setting:
\begin{eqnarray*}
C(S^{N-1}_{\mathbb R,+})&=&C^*\left((x_i)_{i=1,\ldots,N}\Big|x_i=x_i^*,\sum_ix_i^2=1\right)\\
C(T_N^+)&=&C^*\left((x_i)_{i=1,\ldots,N}\Big|x_i=x_i^*,x_i^2=\frac{1}{N}\right)\\
C(O_N^+)&=&C^*\left((u_{ij})_{i,j=1,\ldots,N}\Big|u_{ij}=u_{ij}^*,u^t=u^{-1}\right)
\end{eqnarray*}

Observe that $u_i=\sqrt{N}x_i\in C(T_N^+)$ are subject to the relations $u_i=u_i^*=u_i^{-1}$. Thus, $T_N^+$ appears as the abstract dual of the discrete group $\mathbb Z_2^{*N}$. As for $O_N^+$, this is a compact quantum group in the sense of Woronowicz \cite{wo1}, with structural maps as follows:
$$\Delta(u_{ij})=\sum_ku_{ik}\otimes u_{kj}\quad,\quad\varepsilon(u_{ij})=\delta_{ij}\quad,\quad S(u_{ij})=u_{ji}$$

We refer to the original papers \cite{bgo}, \cite{wan} and to the lecture notes \cite{ba4} for details. In analogy now with what happens in the classical case, we have:
\begin{enumerate}
\item $O_N^+$ is the quantum isometry group of $S^{N-1}_{\mathbb R,+}$, and $S^{N-1}_{\mathbb R,+}$ appears as an homogeneous space over $O_N^+$. In addition, we have an embedding $O_N^+\subset\sqrt{N}\cdot S^{N^2-1}_{\mathbb R,+}$.

\item $T_N^+$ appears inside $S^{N-1}_{\mathbb R,+}$ by setting $x_1^2=\ldots=x_N^2$. Conversely, $S^{N-1}_{\mathbb R,+}$ appears from $T_N^+$ by ``deleting'' this relation, while still keeping $\sum_ix_i^2=1$.

\item $T_N^+\simeq\widehat{\mathbb Z_2^{*N}}$ is a maximal group dual subgroup of $O_N^+$, and $O_N^+$ itself can be reconstructed from this subgroup, via representation theory methods.
\end{enumerate}

To be more precise, having agreed that in the classical case, constructing the $2\times3=6$ correspondences is a good exercise in basic geometry, the situation is similar here, with all this being a good exercise in basic noncommutative geometry. The only point which is non-trivial is the correspondence $T_N^+\to O_N^+$, and we refer here to \cite{ba2}, \cite{bco}, \cite{bve}. Also, we refer to the lecture notes \cite{ba4} for all the needed details on all this material.

Let us try now to understand the possible ``intermediate liberations'' of the usual geometry. At the sphere and cube level, there is a lot of freedom in dealing with this question, or at least the known noncommutative geometry theories here don't provide any simple, quick answer. However, at the quantum group level, things are quite rigid. So, as a first good question, we would like to find the intermediate quantum groups, as follows:
$$\ \qquad \qquad \ \quad O_N\subset G\subset O_N^+\qquad\ \qquad (*)$$

In order to deal with this problem, let us recall Brauer's theorem \cite{bra}. Given a pairing $\pi\in P_2(k,l)$, between an upper row of $k$ points, and a lower row of $l$ points, we set:
$$T_\pi(e_{i_1}\otimes\ldots\otimes e_{i_k})=\sum_{j_1\ldots j_l}\delta_\pi\begin{pmatrix}i_1&\ldots&i_k\\ j_1&\ldots&j_l\end{pmatrix}e_{j_1}\otimes\ldots\otimes e_{j_l}$$

Brauer's theorem states that the intertwining spaces $Hom(u^{\otimes k},u^{\otimes l})$ for the orthogonal group $O_N$ are precisely those spanned by these maps $T_\pi$. In addition, a free version of this result is available, stating that for $O_N^+$, the intertwining spaces are spanned as well by the maps $T_\pi$, but this time with $\pi$ being a noncrossing pairing, $\pi\in NC_2$. See \cite{bco}.

Based on these results, let us call a quantum group $O_N\subset G\subset O_N^+$ easy when the following equalities hold, for a certain category of pairings, $NC_2\subset D\subset P_2$:
$$Hom(u^{\otimes k},u^{\otimes l})=span\left(T_\pi\Big|\pi\in D(k,l)\right)$$

Here the categorical operations are the horizontal and vertical concatenation, and the upside-down turning of the pairings. These operations ensure the fact that the spaces $span\left(T_\pi\big|\pi\in D(k,l)\right)$ form a tensor $C^*$-category, and as a consequence of Woronowicz's Tannakian duality \cite{wo2}, each such category produces a quantum group. For full details regarding the easy quantum group theory, we refer to \cite{bsp}, \cite{mal}, \cite{rwe}.

We are now ready to go back to ($*$). If we restrict the attention to the easy case, we just have to find the intermediate categories $NC_2\subset D\subset P_2$. And here, there is only one solution, namely the category $P_2^*$ generated by the following ``crossing'':
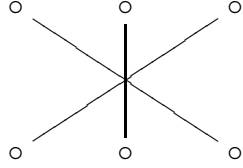
$$\xymatrix@R=15mm@C=10mm{
\circ\ar@{-}[drr]&\circ\ar@{-}[d]&\circ\ar@{-}[dll]
\\
\circ&\circ&\circ}$$

Due to the uniqueness result, we will call this diagram ``half-classical crossing'', and the corresponding relation, namely $abc=cba$, ``half-commutation relation''. See \cite{bve}.

Summarizing, we have now an answer to ($*$), which is in addition unique, in the easy quantum group setting. We should mention that, conjecturally, this solution is actually unique, in the arbitrary compact quantum group setting. See \cite{bbc}.

So, let us go ahead now, and construct our various geometric objects, as follows:
\begin{eqnarray*}
C(S^{N-1}_{\mathbb R,*})&=&C(S^{N-1}_{\mathbb R,+})\Big/\left<abc=cba\Big|\forall a,b,c\in\{x_i\}\right>\\
C(T_N^*)&=&C(T_N^+)\Big/\left<abc=cba\Big|\forall a,b,c\in\{x_i\}\right>\\
C(O_N^*)&=&C(O_N^+)\Big/\left<abc=cba\Big|\forall a,b,c\in\{u_{ij}\}\right>
\end{eqnarray*}

As in the classical and free cases, we have correspondences between these objects, with statements (1-3) as above. For details here, we refer to \cite{ba4}, \cite{bgo}, \cite{bme}, \cite{bve}. 

In view of the uniqueness results mentioned above, we can stop here the axiomatization work, because this is the third and last possible ``geometry''. So, what is left to do now is to start the actual geometric work. We would like  to understand the structure and geometry of the various algebraic manifolds $X\subset S^{N-1}_{\mathbb R,*},X\subset S^{N-1}_{\mathbb R,+}$, along with their differential geometric aspects, and Riemannian aspects as well. We have as well the important question of finding explicit matrix models for the coordinates of such manifolds. The whole subject here is still very young, and we refer to \cite{ba3}, \cite{bbi}, \cite{bic}.

The aim of the present paper is to clarify what happens in the complex geometry setting. Available here are the classical theory, having symmetry group $U_N$, and its free version, with symmetry group $U_N^+$. At the level of intermediate geometries, however, the situation is quite complicated, because we have many quantum groups as follows:
$$\ \qquad \qquad \ \quad U_N\subset G\subset U_N^+\qquad\ \qquad (**)$$

Our goal here will be not that of solving this classification problem for the intermediate complex geometries, but rather of trying to identify the ``main'' solution. For this purpose, we will use an axiomatic approach. We will first axiomatize the triples $(S,T,G)$ which are subject to correspondences (1-3) as above, then we will examine the various complex analogues of the triple $(S^{N-1}_{\mathbb R,*},T_N^*,O_N^*)$, and we will identify the ``main'' solution.

The paper is organized as follows: 1-2 contain various preliminaries and generalities, and in 3-4 we construct and study the complex half-classical geometry.

\section{Formalism}

We agree to call ``noncommutative compact spaces'' the abstract duals of the unital $C^*$-algebras. We denote such spaces by $X,Y,Z,\ldots$, with the corresponding $C^*$-algebras being denoted $C(X),C(Y),C(Z),\ldots$ We use this correspondence for formulating our various findings directly in terms of noncommutative spaces. For instance, we call a morphism $X\to Y$ injective if the corresponding morphism $C(Y)\to C(X)$ in surjective, and vice versa. Also, a direct product $X\times Y$ is by definition the noncommutative space corresponding to the $C^*$-algebra $C(X)\otimes C(Y)$, with $\otimes$ being the minimal tensor product.

We are interested in what follows in the noncommutative analogues of the real algebraic manifolds $X\subset S^{N-1}_\mathbb C$. Here we use of course the canonical embedding $S^{N-1}_\mathbb C\subset\mathbb C^N\simeq\mathbb R^{2N}$, and by real algebraic manifold we mean as usual the set of zeroes of a certain family of polynomials in the standard coordinates on $\mathbb R^N$, or, equivalently, of a certain family of polynomials in the standard coordinates on $\mathbb C^N$, and their conjugates.

Our starting point is the following well-known fact:

\begin{proposition}
Consider a real algebraic manifold $X\subset S^{N-1}_\mathbb C$, appearing as:
$$X=\left\{x\in\mathbb C^N\Big|\sum_i|x_i|^2=1,P_\alpha(x_1,\ldots,x_N)=0\right\}$$
The algebra of continuous functions $f:X\to\mathbb C$ is then given by
$$C(X)=C^*_{comm}\left(x_1,\ldots,x_N\Big|\sum_ix_ix_i^*=1,P_\alpha(x_1,\ldots,x_N)=0\right)$$
where by $C^*_{comm}$ we mean universal commutative $C^*$-algebra.
\end{proposition}

\begin{proof}
Observe first that the universal algebra in the statement is well-defined, because $\sum_ix_ix_i^*=1$ gives $||x_i||\leq1$ for any $i$, and so the biggest norm is bounded. If we denote by $A$ this algebra, we have an arrow $A\to C(X)$. Conversely, by Gelfand duality we have $A=C(X')$ for a certain compact space $X'$. The coordinates $x_i$ produce an embedding $X'\subset\mathbb C^N$, then the condition $\sum_ix_ix_i^*=1$ gives $X'\subset S^{N-1}_\mathbb C$, and finally the conditions $P_\alpha(x_1,\ldots,x_N)=0$ give $X'\subset X$. Thus we have $X=X'$, as claimed.
\end{proof}

The above result suggests to construct a free version $X^+$, simply by removing the commutativity assumption from the presentation of $C(X)$. However, this is quite tricky, because the relations $P_\alpha=0$ must not include, or imply, the commutativity.

In practice, this method works in a number of situations. We have:

\begin{definition}
The free complex sphere and free unitary group are constructed as
\begin{eqnarray*}
C(S^{N-1}_{\mathbb C,+})&=&C^*\left((x_i)_{i=1,\ldots,N}\Big|\sum_ix_ix_i^*=\sum_ix_i^*x_i=1\right)\\
C(U_N^+)&=&C^*\left((u_{ij})_{i=1,\ldots,N}\Big|u^*=u^{-1},u^t=\bar{u}^{-1}\right)
\end{eqnarray*}
where on the right we have universal $C^*$-algebras.
\end{definition}

As explained by Wang in \cite{wan}, the above noncommutative space $U_N^+$ is a compact quantum group in the sense of Woronowicz \cite{wo1}, with structural maps as follows:
$$\Delta(u_{ij})=\sum_ku_{ik}\otimes u_{kj}\quad,\quad\varepsilon(u_{ij})=\delta_{ij}\quad,\quad S(u_{ij})=u_{ji}^*$$

We have an action $U_N^+\curvearrowright S^{N-1}_{\mathbb C,+}$, whose properties are quite similar to that of the action $U_N\curvearrowright S^{N-1}_\mathbb C$. In order to explain this material, let us introduce a few more notions:

\begin{definition}
Consider an algebraic submanifold $X\subset S^{N-1}_{\mathbb C,+}$, i.e. a closed subset defined via algebraic relations, and a closed quantum subgroup $G\subset U_N^+$.
\begin{enumerate}
\item We say that we have an affine action $G\curvearrowright X$ when the formula $x_i\to\sum_ju_{ij}\otimes x_j$ defines a morphism of $C^*$-algebras $\Phi:C(X)\to C(G)\otimes C(X)$.

\item The biggest quantum subgroup $G\subset U_N^+$ acting affinely on $X$ is denoted $G^+(X)$, and is called quantum isometry group of $X$.
\end{enumerate}
\end{definition}

Here by ``algebraic relations'' we mean of course relations of type $P_\alpha(x_1,\ldots,x_N)=0$, with $P_\alpha$ being noncommutative $*$-polynomials in $N$ variables. As for the word ``biggest'', this means ``maximal in the appropriate category''. We agree in what follows to keep using the language of noncommutative compact spaces and manifolds, and to use more complicated language only when needed, and helpful in connection with problems.  

Observe that the morphism in (1) above is automatically coassociative, $(\Phi\otimes id)\Phi=(id\otimes\Delta)\Phi$, and counital as well, $(id\otimes\varepsilon)\Phi=id$. When $X,G$ are both classical such a morphism must appear by transposition from a usual affine group action $G\times X\to X$.

Regarding (2), it is routine to check that such a biggest quantum group exists indeed, simply by dividing $C(U_N^+)$ by the appropriate relations. We refer here to \cite{ba2} for details. 

We should mention that, in analogy with what happens in the classical case, there are of course several notions of quantum isometries, and quantum isometry groups, and those presented above are those that we will need here.  As an example, consider the usual sphere $S^{N-1}_\mathbb R$. The classical isometries of $S^{N-1}_\mathbb R$ are ``obviously'' the orthogonal matrices $U\in O_N$, but the meaning of ``obvious'' can of course vary with the person involved:
\begin{enumerate}
\item If we agree that the sphere is ``something round'', which is basic common sense, the isometry group is quite complicated to compute. We obtain $O_N$.

\item If we agree that the sphere is the set of solutions of $\sum_ix_i^2=1$, which is perhaps a bit unnatural, the isometry group is easy to compute. We obtain $O_N$ too.
\end{enumerate}

In the noncommutative setting, all this becomes considerably more complicated. For a full discussion on these topics, we refer to Goswami's papers \cite{go1}, \cite{go2}.

We will need as well the following constructions:

\begin{definition}
Consider a subspace $S\subset S^{N-1}_{\mathbb C,+}$, and a subgroup $G\subset U_N^+$.
\begin{enumerate}
\item The standard torus of $S$ is the subspace $T\subset S$ obtained by setting, at the algebra level, $C(T)=C(S)/<x_ix_i^*=x_i^*x_i=1/N>$.

\item The diagonal torus of $G$ is the subspace $\mathcal T\subset G$ obtained by setting, at the algebra level, $C(\mathcal T)=C(G)/<u_{ij}=0|\forall i\neq j>$.
\end{enumerate}
\end{definition}

Here, as usual, we use the direct language of noncommutative compact spaces, with ``subspace'' standing for ``closed noncommutative subspace'' and ``subgroup'' standing for ``closed quantum subgroup'', with all this coming from the Gelfand theorem.

For $S=S^{N-1}_\mathbb C$, the space constructed in (1) is the usual torus, $T=\mathbb T^N$. For $S=S^{N-1}_{\mathbb C,+}$, the rescaled generators $u_i=\sqrt{N}x_i$ are subject to the relations $u_i^*=u_i^{-1}$, which produce the free group algebra $C^*(F_N)$. Thus, we obtain here a group dual, $T=\widehat{F_N}$.

Regarding (2), observe that $C(\mathcal T)$ is generated by the variables $g_i=u_{ii}$, which are group-like. Thus $\mathcal T=\widehat{\Gamma}$, where $\Gamma$ is the group generated by these variables. See \cite{bve}.

We can now formulate our main definition, as follows:

\begin{definition}
A noncommutative geometry consists of three spaces, 
\begin{enumerate}
\item an intermediate algebraic manifold $S^{N-1}_\mathbb R\subset S\subset S^{N-1}_{\mathbb C,+}$, called sphere,

\item an intermediate space $\widehat{\mathbb Z_2^N}\subset T\subset\widehat{F_N}$, called torus,

\item and an intermediate quantum group, $O_N\subset G\subset U_N^+$,
\end{enumerate}
such that the following two conditions are satisfied,
\begin{enumerate}
\item $G$ is the quantum isometry group of $S$,

\item $T$ is the standard torus of $S$, as well as the diagonal torus of $G$,
\end{enumerate}
where, for the needs of the second axiom, we think of full and reduced group algebras as representing the same quantum space.
\end{definition}

This definition is something technical, and temporary. Further improving it, and comparing it with the other known definitions of noncommutative manifolds and geometries, and looking for applications too, is of course something that we have in mind. In what follows we will focus on this definition as it is, and present our results on the subject, as they are, and we will be of course back to all this in our future papers.

Regarding the terminology, the torus $T$ will be sometimes called ``cube'' of the geometry, because in the real case, as explained below, what we have here is rather a cube. Also, it is also useful to think of the discrete dual $\Gamma=\widehat{T}$, which appears as an intermediate quotient group $F_N\to\Gamma\to\mathbb Z_2^N$, as being the ``structural group'' of the geometry.

Finally, regarding the identification made at the end, this is something standard in noncommutative geometry. The point indeed is that there is a ``problem'' with Gelfand duality, coming from the fact that a discrete group dual $T=\widehat{\Gamma}$ can be represented by several $C^*$-algebras, including the maximal one $C^*(\Gamma)$ and the minimal one $C^*_{red}(\Gamma)$. The standard way of fixing this issue is that of identifying all these algebras, and this is what we do here too. For a full discussion on all this, see Woronowicz \cite{wo1}.

Here is now a useful reformulation of the axioms, in terms of $S$ only:

\begin{proposition}
An algebraic manifold $S^{N-1}_\mathbb R\subset S\subset S^{N-1}_{\mathbb C,+}$, with standard torus denoted $T\subset S$, produces a noncommutative geometry precisely when:
\begin{enumerate}
\item We have an affine action $O_N\curvearrowright S$.

\item $\delta(x_i)=\sqrt{N}x_i\otimes x_i$ defines a morphism of algebras $\delta:C(S)\to C(T)\otimes C(S)$.
\end{enumerate}
If these conditions are satisfied, we say that $S$ is a noncommutative sphere.
\end{proposition}

\begin{proof}
Given $S^{N-1}_\mathbb R\subset S\subset S^{N-1}_{\mathbb C,+}$, consider its quantum isometry group $G\subset U_N^+$, and consider as well the standard torus $T\subset S$, and the diagonal torus $\mathcal T\subset G$.

Assuming that $(S,T,G)$ is a noncommutative geometry, we have $O_N\subset G$, so the condition (1) is clear. Also, since we have $T=\mathcal T$, the morphism in (2) simply appears by composing the universal coaction with the diagonal torus quotient map:
$$\begin{matrix}
C(S)&\to&C(G)\otimes C(S)&\to&C(\mathcal T)\otimes C(S)&=&C(T)\otimes C(S)\\
x_i&\to&\sum_ju_{ij}\otimes x_j&\to&g_i\otimes x_i&=&\sqrt{N}x_i\otimes x_i
\end{matrix}$$

Conversely, assuming that (1,2) are satisfied, we must prove that we have $T=\mathcal T$. For this purpose, observe first that the map $\delta$ induces a morphism as follows:
$$\overline{\delta}:C(T)\to C(T)\otimes C(T)\quad,\quad x_i\to\sqrt{N} x_i \otimes x_i$$

Thus $C(T)$ is a cocommutative Hopf algebra, and its elements $\sqrt{N}x_i$ are group-like. 

With this picture in mind, the map $\delta$ in the statement corresponds to an action $T\curvearrowright S$, and by universality of the quantum isometry group $G$, we obtain a map as follows:
$$\theta:C(\mathcal T)\to C(T)\quad,\quad u_{ii}\to\sqrt{N} x_i$$

In order to construct an inverse map, we will use a composition, as follows: 
$$\begin{matrix}
C(S)&\to&C(G)\otimes C(S)&\to&C(G)&\to&C(G)&\to&C(\mathcal T)\\
x_i&\to&\sum_ju_{ij}\otimes x_j&\to&u_{ij}&\to&\frac{1}{\sqrt{N}}\sum_ju_{ij}&\to&\frac{1}{\sqrt{N}}u_{ii}
\end{matrix}$$

Here the first map is the universal coaction, and the second map comes from the evaluation at $(1,0,\ldots,0)\in S^{N-1}_\mathbb R\subset S$, which gives a map $\varepsilon:C(S)\to\mathbb C$, $x_i\to\delta_{i1}$. Regarding the third map, this is the algebra automorphism $u_{ij}\to\sum_km_{kj}u_{ik}$ constructed by using the inclusion $O_N\subset G$ coming from our assumption (1), along with an orthogonal matrix $M=(m_{ij})$ having $\frac{1}{\sqrt{N}}$ entries on its first column. Observe that such orthogonal matrices $M$ exist indeed, for instance by using the Gram-Schmidt orthogonalization procedure. Finally, the fourth map is the canonical quotient map.

Now since the elements $X_i=\frac{1}{\sqrt{N}}u_{ii}$ satisfy the relations  $X_iX_i^*=X_i^*X_i=\frac{1}{N}$, the composition that we constructed factorizes into a map as follows:
$$\rho:C(T)\to C(\mathcal T)\quad,\quad x_i\to\frac{1}{\sqrt{N}}u_{ii}$$

It is clear that $\theta,\rho$ are inverse to each other, and this finishes the proof.
\end{proof}

At the level of basic examples, we have:

\begin{proposition}
We have the following examples of geometries:
\begin{enumerate}
\item Real geometry: $S=S^{N-1}_\mathbb R$, $T=\widehat{\mathbb Z_2^N}$, $G=O_N$.

\item Complex geometry: $S=S^{N-1}_\mathbb C$, $T=\widehat{\mathbb Z^N}$, $G=U_N$.

\item Free real geometry: $S=S^{N-1}_{\mathbb R,+}$, $T=\widehat{\mathbb Z_2^{*N}}$, $G=O_N^+$.

\item Free complex geometry: $S=S^{N-1}_{\mathbb C,+}$, $T=\widehat{F_N}$, $G=U_N^+$.
\end{enumerate}
\end{proposition}

\begin{proof}
These results are well-known, and we refer to \cite{ba1}, \cite{bgo} for details here.
\end{proof}

Based on these examples, here are now a few more general notions:

\begin{definition}
A noncommutative geometry is called:
\begin{enumerate}
\item Real, if $S\subset S^{N-1}_{\mathbb R,+}$, $T\subset\widehat{\mathbb Z_2^{*N}}$, $G\subset O_N^+$.

\item Complex, if $S^{N-1}_\mathbb C\subset S$, $\widehat{\mathbb Z^N}\subset T$, $U_N\subset G$.

\item Classical, if $S\subset S^{N-1}_\mathbb C$, $T\subset\widehat{\mathbb Z^N}$, $G\subset U_N$.

\item Free, if $S^{N-1}_{\mathbb R,+}\subset S$, $\widehat{\mathbb Z_2^{*N}}\subset T$, $O_N^+\subset G$.
\end{enumerate}
\end{definition}

We will illustrate these notions in what follows, with several other examples.

Let us introduce now a few more notions. First, a geometry which is not real, nor complex, will be called ``hybrid''. Also, a geometry which is not classical, nor free, will be called ``intermediate''. The basic examples here come from:

\begin{definition}
We have spheres, tori, and quantum groups, as follows:
\begin{enumerate}
\item $S^{N-1}_{\mathbb R,*},\widehat{\mathbb Z_2^{\circ N}},O_N^*$, obtained respectively from $S^{N-1}_{\mathbb R,+},\widehat{\mathbb Z_2^{*N}},O_N^+$ by imposing to the standard coordinates the relations $abc=cba$.

\item $S^{N-1}_{\mathbb C,**},\widehat{\mathbb Z^{\circ\circ N}},U_N^{**}$, obtained respectively from $S^{N-1}_{\mathbb C,+},\widehat{F_N},U_N^+$ by imposing to the standard coordinates and their adjoints the relations $abc=cba$.
\end{enumerate}
\end{definition}

Here (1) is a well-established definition, coming from the work in \cite{bsp}, \cite{bve}. Regarding (2), we have there some temporary objects, coming from \cite{bdu}, which are somehow ``minimal'', and which will be replaced with the correct, maximal ones, later on.

We can now formulate our first result, as follows:

\begin{theorem}
We have geometries, whose unitary groups are as follows,
$$\xymatrix@R=15mm@C=15mm{
U_N\ar[r]&U_N^{**}\ar[r]&U_N^+\\
\mathbb TO_N\ar[r]\ar[u]&\mathbb TO_N^*\ar[r]\ar[u]&\mathbb TO_N^+\ar[u]\\
O_N\ar[r]\ar[u]&O_N^*\ar[r]\ar[u]&O_N^+\ar[u]}$$
with the middle row spaces obtained from the upper ones via the relations $ab^*=a^*b$.
\end{theorem}

\begin{proof}
The results in the upper and lower row are well-known, see for instance \cite{ba1}. Regarding the middle row, consider indeed the following quotient algebra:
$$C(\mathbb TO_N^+)=C(U_N^+)\Big/\left<ab^*=a^*b\Big|\forall a,b\in\{u_{ij}\}\right>$$

Inside this algebra, observe that with $U_{ij}=\sum_au_{ia}\otimes u_{aj}$ we have:
$$U_{ij}U_{kl}^*=\sum_{ab}u_{ia}u_{kb}^*\otimes u_{aj}u_{bl}^*=\sum_{ab}u_{ia}^*u_{kb}\otimes u_{aj}^*u_{bl}=U_{ij}^*U_{kl}$$

Thus we can construct a comultiplication morphism $\Delta$, by mapping $u_{ij}\to U_{ij}$, and the existence of the counit $\varepsilon$ and of the antipode $S$ is clear too. Now with $\mathbb TO_N^+$ constructed as above, we can construct the other quantum groups as well, as follows:
$$\mathbb TO_N=\mathbb TO_N^+\cap U_N\quad,\quad\mathbb TO_N^*=\mathbb TO_N^+\cap U_N^{**}$$

For the spheres and tori, the discussion here parallels the one from the quantum group case. To be more precise, the definition of these objects is as follows:
\begin{eqnarray*}
C(\mathbb TS^{N-1}_{\mathbb R,\times})&=&C(S^{N-1}_{\mathbb C,\times})\Big/\left<ab^*=a^*b\Big|\forall a,b\in\{x_i\}\right>\\
C(\mathbb T\widehat{\mathbb Z_2^{\times N}})&=&C(\widehat{\mathbb Z_2^{\times N}})\Big/\left<ab^*=a^*b\Big|\forall a,b\in\{g_i\}\right>
\end{eqnarray*}

Regarding the axioms, let us prove now that the standard action $U_N^\times\curvearrowright S^{N-1}_{\mathbb C,\times}$ restricts to an action $\mathbb TO_N^\times\curvearrowright\mathbb TS^{N-1}_{\mathbb R,\times}$. With $X_i=\sum_au_{ia}\otimes x_a$ we have:
$$X_iX_j^*=\sum_{ab}u_{ia}u_{jb}^*\otimes x_ax_b^*=\sum_{ab}u_{ia}^*u_{jb}\otimes x_a^*x_b=X_i^*X_j$$

Thus we can indeed define our coaction map, via $x_i\to X_i$. In order to prove now the universality, assume that we have an action $G\curvearrowright\mathbb TS^{N-1}_{\mathbb R,\times}$. With $X_i=\sum_au_{ia}\otimes x_a$ as above we have $X_iX_j^*=X_i^*X_j$, and from this we obtain, via the standard method from \cite{bhg}, that we have $u_{ia}u_{jb}^*=u_{ia}^*u_{jb}$ for any $i,j,a,b$, and so $G\subset\mathbb TO_N^\times$. We refer here to \cite{ba1} for the classical case, the proof in the half-classical and free cases being similar.
\end{proof}

\section{Easiness, amenability}

We will need the notion of easy quantum group, from \cite{bsp}, \cite{fre}, \cite{rwe}, \cite{twe}.

We denote by $P(k,l)$ the set of partitions between an upper row of $k$ points, and a lower row of $l$ points, with each leg colored black or white, and with $k,l$ standing for the corresponding ``colored integers''. We have the following notion:

\begin{definition}
A category of partitions is a collection of sets $D=\bigcup_{kl}D(k,l)$, with $D(k,l)\subset P(k,l)$, which contains the identity, and is stable under:
\begin{enumerate}
\item The horizontal concatenation operation $\otimes$.

\item The vertical concatenation $\circ$, after deleting closed strings in the middle.

\item The upside-down turning operation $*$ (with reversing of the colors).
\end{enumerate}
\end{definition}

As explained in \cite{twe}, such categories produce quantum groups. To be more precise, associated to any partition $\pi\in P(k,l)$ is the following linear map:
$$T_\pi(e_{i_1}\otimes\ldots\otimes e_{i_k})=\sum_{j_l\ldots j_l}\delta_\pi\begin{pmatrix}i_1&\ldots&i_k\\ j_1&\ldots&j_l\end{pmatrix}e_{j_1}\otimes\ldots\otimes e_{j_l}$$

Here the Kronecker type symbol $\delta_\pi\in\{0,1\}$ is by definition 1 if all the strings of $\pi$ join pairs of equal indices, and is 0 otherwise. With this notion in hand, we have:

\begin{definition}
A compact quantum group $G\subset U_N^+$ is called easy when
$$Hom(u^{\otimes k},u^{\otimes l})=span\left(T_\pi\Big|\pi\in D(k,l)\right)$$
for any $k,l$, for a certain category of partitions $D\subset P$.
\end{definition}

In other words, the easiness condition states that Tannakian dual of $G$, also called Schur-Weyl dual, should come in the simplest possible way: from partitions.

In order to discuss some basic examples, consider the categories of pairings, and of noncrossing pairings, $NC_2\subset P_2$. Consider as well the ``color-matching'' versions of these categories, $\mathcal{NC}_2\subset \mathcal P_2$, the color-matching condition stating that the various types of strings (upper, lower, through) of our pairings must be colored as follows:
$$
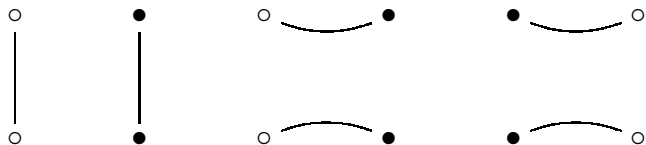
\xymatrix@R=12mm@C=12mm{
\circ\ar@{-}[d]&\bullet\ar@{-}[d]&\circ\ar@{-}@/_/[r]&\bullet&\bullet\ar@{-}@/_/[r]&\circ\\
\circ&\bullet&\circ\ar@{-}@/^/[r]&\bullet&\bullet\ar@{-}@/^/[r]&\circ}$$

With these notions in hand, we have the following result:

\begin{proposition}
We have easy quantum groups, as follows,
$$\xymatrix@R=15mm@C=15mm{
U_N\ar[r]&U_N^+\\
O_N\ar[r]\ar[u]&O_N^+\ar[u]}
\ \ \xymatrix@R=9mm@C=5mm{\\ \ar@{~}[r]&\\&\\}\ \ 
\xymatrix@R=16mm@C=16mm{
\mathcal P_2\ar[d]&\mathcal{NC}_2\ar[l]\ar[d]\\
P_2&NC_2\ar[l]}$$
with the diagram at right describing the corresponding categories of partitions.
\end{proposition}

\begin{proof}
Here the results on the right are the Brauer theorem for $O_N$, and for $U_N$, and the results on the left are free versions of Brauer's theorem, discussed in \cite{bco}. As a quick, partly heuristic explanation here, all these results follow from Tannakian duality:

(1) $U_N^+$ is defined via the relations $u^*=u^{-1},u^t=\bar{u}^{-1}$, which tell us that the operators $T_\pi$, with $\pi={\ }^{\,\cap}_{\circ\bullet}$ and $\pi={\ }^{\,\cap}_{\bullet\circ}$, must be in the associated Tannakian category $C$. We therefore obtain $C=span(T_\pi|\pi\in D)$, with $D=<{\ }^{\,\cap}_{\circ\bullet}\,\,,{\ }^{\,\cap}_{\bullet\circ}>={\mathcal NC}_2$, as claimed.

(2) $O_N^+\subset U_N^+$ is defined by imposing the relations $u_{ij}=\bar{u}_{ij}$, which tell us that the operators $T_\pi$, with $\pi=|^{\hskip-1.32mm\circ}_{\hskip-1.32mm\bullet}$ and $\pi=|_{\hskip-1.32mm\circ}^{\hskip-1.32mm\bullet}$, must be in the associated Tannakian category $C$. We therefore obtain $C=span(T_\pi|\pi\in D)$, with $D=<\mathcal{NC}_2,|^{\hskip-1.32mm\circ}_{\hskip-1.32mm\bullet},|_{\hskip-1.32mm\circ}^{\hskip-1.32mm\bullet}>=NC_2$, as claimed.

(3) $U_N\subset U_N^+$ is defined via the relations $[u_{ij},u_{kl}]=0$ and $[u_{ij},\bar{u}_{kl}]=0$, which tell us that the operators $T_\pi$, with $\pi={\slash\hskip-2.1mm\backslash}^{\hskip-2.5mm\circ\circ}_{\hskip-2.5mm\circ\circ}$ and $\pi={\slash\hskip-2.1mm\backslash}^{\hskip-2.5mm\circ\bullet}_{\hskip-2.5mm\bullet\circ}$, must be in the associated Tannakian category $C$. Thus $C=span(T_\pi|\pi\in D)$, with $D=<\mathcal{NC}_2,{\slash\hskip-2.1mm\backslash}^{\hskip-2.5mm\circ\circ}_{\hskip-2.5mm\circ\circ},{\slash\hskip-2.1mm\backslash}^{\hskip-2.5mm\circ\bullet}_{\hskip-2.5mm\bullet\circ}>=\mathcal P_2$, as claimed.

(4) Finally, in order to deal with $O_N$, we can use here the formula $O_N=O_N^+\cap U_N$. At the categorical level, this tells us that the associated Tannakian category is given by $C=span(T_\pi|\pi\in D)$, with $D=<NC_2,\mathcal P_2>=P_2$, as claimed.
\end{proof}

There are many other examples of easy quantum groups, as for instance the permutation group $S_N$ and its free analogue $S_N^+$, the corresponding categories being here the category of all partitions $P$, and the category of noncrossing partitions $NC\subset P$. See \cite{bsp}.

In the pairing case, however, the main examples remain those in Proposition 2.3. Now observe that the 4 quantum groups there are precisely the unitary groups of the 4 main geometries, from Proposition 1.7. We are therefore led to the following notion:

\begin{definition}
A noncommutative geometry is called easy when its associated unitary group $O_N\subset G\subset U_N^+$ is easy, coming from a category of pairings ${\mathcal NC}_2\subset D\subset P_2$.
\end{definition}

Our first task will be that of proving that the 9 geometries from Theorem 1.10 above are all easy. For this purpose, let $P_2^*\subset P_2$ be the category of pairings having the property that when flatenning the pairing (which means rotating, as for the resulting pairing to have only lower legs), each string has an even number of points between its legs. Let also $\bar{P}_2\subset P_2$ be the category of pairings having the property that when flatenning the pairing, the number of $\circ$ symbols equals the number of $\bullet$ symbols. Finally, let us define as well categories $\mathcal P_2^{**},\bar{P}_2^*,\bar{NC}_2$ in the obvious way, by taking intersections.

With these notions in hand, we have the following result:

\begin{proposition}
The basic $9$ geometries are easy, with quantum groups as follows,
$$
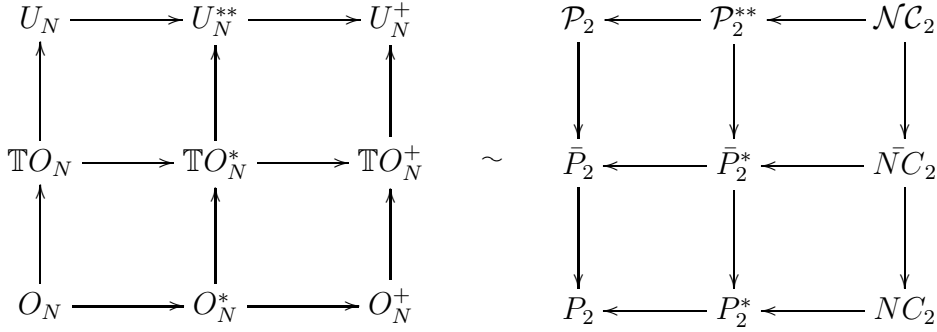
\xymatrix@R=12mm@C=12mm{
U_N\ar[r]&U_N^{**}\ar[r]&U_N^+\\
\mathbb TO_N\ar[r]\ar[u]&\mathbb TO_N^*\ar[r]\ar[u]&\mathbb TO_N^+\ar[u]\\
O_N\ar[r]\ar[u]&O_N^*\ar[r]\ar[u]&O_N^+\ar[u]}
\ \ \xymatrix@R=16mm@C=5mm{\\ \ar@{~}[r]&\\&\\}\ \ 
\xymatrix@R=12.7mm@C=12.7mm{
\mathcal P_2\ar[d]&\mathcal P_2^{**}\ar[l]\ar[d]&\mathcal{NC}_2\ar[l]\ar[d]\\
\bar{P}_2\ar[d]&\bar{P}_2^*\ar[l]\ar[d]&\bar{NC}_2\ar[l]\ar[d]\\
P_2&P_2^*\ar[l]&NC_2\ar[l]}$$
with the diagram at right describing the corresponding categories of partitions.
\end{proposition}

\begin{proof}
The idea is that of converting the defining relations for the quantum group into statements regarding certain operators of type $T_\pi$, and then computing the category of partitions generated by these defining partitions $\pi$. More precisely:

(1) First of all, the 4 results at the corners are known from Proposition 2.3. Also, the result for $O_N^*$ is known from \cite{bve}, and its unitary version, for $U_N^{**}$, is known from \cite{ba1}. We are therefore left with proving the 3 results corresponding to the middle rows.

(2) As a first observation here, the result is clear for $\mathbb TO_N$, and this, by using an elementary approach. Indeed, if we denote the standard corepresentation by $u=zv$, with $z\in\mathbb T$ and with $v=\bar{v}$, then in order to have $Hom(u^{\otimes k},u^{\otimes l})\neq\emptyset$, the $z$ variabes must cancel, and in the case where they cancel, we obtain the same Hom-space as for $O_N$. Now since the cancelling property for the $z$ variables corresponds precisely to the fact that $k,l$ must have the same numbers of $\circ$ symbols minus $\bullet$ symbols, the associated Tannakian category must come from the category of pairings $\bar{P}_2\subset P_2$, as claimed.

(3) In order to deal now with the free version $\mathbb TO_N^+$, no such shortcut is available here, and we must use the regular, abstract method. So, observe that the defining relations for this quantum group, namely $ab^*=a^*b$, correspond to the following diagram:
$$
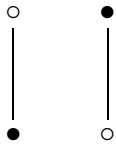
\xymatrix@R=12mm@C=8mm{
\circ\ar@{-}[d]&\bullet\ar@{-}[d]\\
\bullet&\circ}$$

Thus the associated category of partitions is $D=<\mathcal{NC}_2,|^{\hskip-1.32mm\circ}_{\hskip-1.32mm\bullet}|_{\hskip-1.32mm\circ}^{\hskip-1.32mm\bullet}>=\bar{NC}_2$, as claimed. 

(4) Finally, since we have $\mathbb TO_N^*=\mathbb TO_N^+\cap U_N^{**}$, here the associated category of partitions follows to be $D=<\bar{NC}_2,\mathcal P_2^{**}>=\bar{P}_2^*$, and this finishes the proof.
\end{proof}

Now back to the general case, our claim is that, for an easy geometry, there are a few simplifications in the axioms. We first have the following result:

\begin{proposition}
For a geometry which is easy, coming from a category of pairings ${\mathcal NC}_2\subset D\subset P_2$, the associated quantum group is given by
$$C(G)=C(U_N^+)\Big/\left<T_\pi\in Hom(u^{\otimes k},u^{\otimes l})\Big|\forall k,l,\forall\pi\in D(k,l)\right>$$
and the associated noncommutative torus is $T=\widehat{\Gamma}$, with:
$$\Gamma=F_N\Big/\left<g_{i_1}\ldots g_{i_k}=g_{j_1}\ldots g_{j_l}\Big|\forall i,j,k,l,\exists\pi\in D(k,l),\delta_\pi\begin{pmatrix}i\\ j\end{pmatrix}\neq0\right>$$
Moreover, in both cases, we can just use partitions $\pi$ which generate the category $D$.
\end{proposition}

\begin{proof}
The first assertion is well-known, see \cite{bsp}, \cite{mal}. If we denote by $g_i=u_{ii}$ the standard coordinates on the associated torus $T$, then we have, with $g=diag(g_1,\ldots,g_N)$:
\begin{eqnarray*}
C(T)
&=&\left[C(U_N^+)\Big/\left<T_\pi\in Hom(u^{\otimes k},u^{\otimes l})\Big|\forall\pi\in D\right>\right]\Big/\left<u_{ij}=0\Big|\forall i\neq j\right>\\
&=&\left[C(U_N^+)\Big/\left<u_{ij}=0\Big|\forall i\neq j\right>\right]\Big/\left<T_\pi\in Hom(u^{\otimes k},u^{\otimes l})\Big|\forall\pi\in D\right>\\
&=&C^*(F_N)\Big/\left<T_\pi\in Hom(g^{\otimes k},g^{\otimes l})\Big|\forall\pi\in D\right>
\end{eqnarray*}

The associated discrete group, $\Gamma=\widehat{T}$, is therefore given by:
$$\Gamma=F_N\Big/\left<T_\pi\in Hom(g^{\otimes k},g^{\otimes l})\Big|\forall\pi\in D\right>$$

Now observe that, with $g=diag(g_1,\ldots,g_N)$, we have:
\begin{eqnarray*}
T_\pi g^{\otimes k}(e_{i_1}\otimes\ldots\otimes e_{i_k})&=&\sum_{j_1\ldots j_l}\delta_\pi\begin{pmatrix}i_1&\ldots&i_k\\ j_1&\ldots&j_l\end{pmatrix}e_{j_1}\otimes\ldots\otimes e_{j_l}\cdot g_{i_1}\ldots g_{i_k}\\
g^{\otimes l}T_\pi(e_{i_1}\otimes\ldots\otimes e_{i_k})&=&\sum_{j_1\ldots j_l}\delta_\pi\begin{pmatrix}i_1&\ldots&i_k\\ j_1&\ldots&j_l\end{pmatrix}e_{j_1}\otimes\ldots\otimes e_{j_l}\cdot g_{j_1}\ldots g_{j_l}
\end{eqnarray*}

We conclude that the relation $T_\pi\in Hom(g^{\otimes k},g^{\otimes l})$ reformulates as follows:
$$\delta_\pi\begin{pmatrix}i_1&\ldots&i_k\\ j_1&\ldots&j_l\end{pmatrix}\neq0\implies g_{i_1}\ldots g_{i_k}=g_{j_1}\ldots g_{j_l}$$

Thus we obtain the formula in the statement. Finally, the last assertion follows from Tannakian duality in the quantum group case, and then in the torus case as well.
\end{proof}

We conjecture that, in the case of an easy geometry, the category $D$ determines everything, and is determined by everything. Thus, in this case we should have full correspondences, between all the objects involved, with $D$ being the central object:
$$
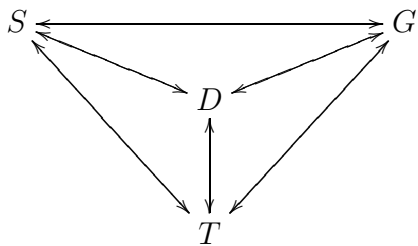
\xymatrix@R=5mm@C=20mm{
S\ar[rr]\ar[dr]\ar[dddr]&&G\ar[dl]\ar[dddl]\ar[ll]\\
&D\ar[dd]\ar[ul]\ar[ur]&\\
&&\\
&T\ar[uu]\ar[uuul]\ar[uuur]&}$$

Observe the similarity with the usual real and complex geometries, which have as well a central object, namely $\mathbb R^N,\mathbb C^N$. Thus, in a certain abstract sense, for an easy geometry, $D$ is the analogue of the ambient space $\mathbb R^N,\mathbb C^N$, which cannot be axiomatized. 

Here are now a few more abstract notions:

\begin{definition}
A noncommutative geometry is called:
\begin{enumerate}
\item Amenable, when the discrete quantum group $\widehat{G}$ is amenable.

\item Weakly amenable, when the discrete group $\Gamma=\widehat{T}$ is amenable.
\end{enumerate}
\end{definition}

Here we use the usual amenability notion for the discrete groups, and for the discrete quantum groups. For the general theory regarding this latter notion, see \cite{ntu}.

We should mention that, for all the known examples of noncommutative geometries in our sense, the coamenability of $G$ is equivalent to the coamenability of $T$. We conjecture that this should be true in general, but have no idea on how to prove this.

Let us discuss now more in detail the half-classical real geometry, associated to $O_N^*$. We recall that the projective version of a quantum subgroup $G\subset O_N^+$ is the quotient quantum group $G\to PG$ having $w_{ij,kl}=u_{ik}u_{jl}$ as fundamental corepresentation. In the classical case, $G\subset O_N$, we recover in this way the usual projective version.

We have the following uniqueness results:

\begin{theorem}
We have the following results:
\begin{enumerate}
\item $O_N^*$ is the unique easy quantum group $O_N\subset G\subset O_N^+$.

\item $O_N^*$ is maximal coamenable, in the easy framework.

\item $O_N^*$ is the biggest quantum group whose projective version $PO_N^*$ is classical.
\end{enumerate}
\end{theorem}

\begin{proof}
These results are well-known, but we remind here the ideas of the proofs, because these will serve as inspiration for various unitary generalizations, to be done below:

(1) This result is from \cite{bve}, the idea being that $P_2^*$ is the unique intermediate category of partitions $NC_2\subset D\subset P_2$. We should mention here that, conjecturally, $O_N^*$ is the unique quantum group $O_N\subset G\subset O_N^+$, even without the easiness assumption. See \cite{bbc}.

(2) The precise claim here is that $O_N^*$ is coamenable, and is in addition maximal with this coamenability property, in the easy quantum group framework. Regarding the coamenability, this is known from \cite{bve}. As for the maximality claim, this follows from (1).

(3) This is well-known as well, since \cite{bve}. Indeed, the relations $abc=cba$ are equivalent to the relations $abcd=cdab$, as shown by the following two computations:
$$\left[abc=cba\right]\implies\left[abcd=cbad=cdab\right]$$
$$\left[abcd=cdab\right]\implies\left[abc=\sum_dabcdd=\sum_dcdabd=\sum_dcbdda=cba\right]$$

Here we assume that all the variables are standard coordinates, and we have used the quadratic condition relating these coordinates, namely $\sum_dd^2=1$.
\end{proof}

\section{Complex geometries}

In this section and in the next one we discuss the construction of the complex half-classical geometry. We will proceed in two steps:

(1) In this section we discuss a first extension of the $U_N^{**}$ geometry, with unitary group denoted $U_N^\times$. This quantum group is the one constructed in \cite{bdd}, \cite{bd+}, and denoted $U_N^*$ there. The problem, however, is that this geometry is not amenable.

(2) In section 4 we construct and study the ``correct'' complex analogue of the $O_N^*$ geometry, with unitary group denoted $U_N^*$. This quantum group is the one constructed in \cite{bbi}, and denoted $U_{N,\infty}$ there. The enlarged picture will look as follows:
$$
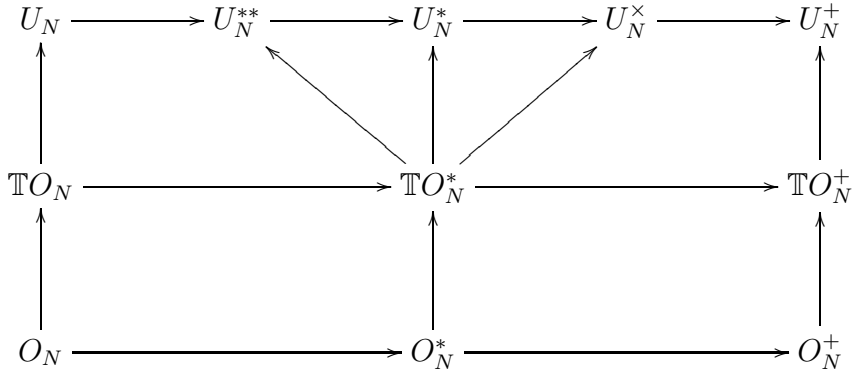
\xymatrix@R=15mm@C=16mm{
U_N\ar[r]&U_N^{**}\ar[r]&U_N^*\ar[r]&U_N^\times\ar[r]&U_N^+\\
\mathbb TO_N\ar[rr]\ar[u]&&\mathbb TO_N^*\ar[rr]\ar[u]\ar[ul]\ar[ur]&&\mathbb TO_N^+\ar[u]\\
O_N\ar[rr]\ar[u]&&O_N^*\ar[rr]\ar[u]&&O_N^+\ar[u]}$$

In order to get started, our first task is to look for extensions of the $U_N^{**}$ geometry. This geometry is by definition easy, coming from the following diagrams:
$$
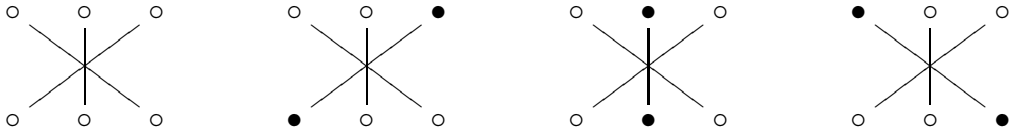
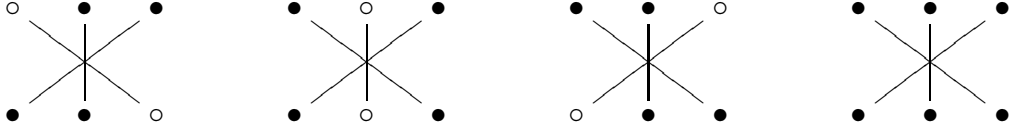
\xymatrix@R=10mm@C=5mm{
\circ\ar@{-}[drr]&\circ\ar@{-}[d]&\circ\ar@{-}[dll]
\\
\circ&\circ&\circ}\quad\ \qquad
\xymatrix@R=10mm@C=5mm{
\circ\ar@{-}[drr]&\circ\ar@{-}[d]&\bullet\ar@{-}[dll]
\\
\bullet&\circ&\circ}\quad\ \qquad
\xymatrix@R=10mm@C=5mm{
\circ\ar@{-}[drr]&\bullet\ar@{-}[d]&\circ\ar@{-}[dll]
\\
\circ&\bullet&\circ}\quad\ \qquad
\xymatrix@R=10mm@C=5mm{
\bullet\ar@{-}[drr]&\circ\ar@{-}[d]&\circ\ar@{-}[dll]
\\
\circ&\circ&\bullet}$$

$$\xymatrix@R=10mm@C=5mm{
\circ\ar@{-}[drr]&\bullet\ar@{-}[d]&\bullet\ar@{-}[dll]
\\
\bullet&\bullet&\circ}\quad\ \qquad
\xymatrix@R=10mm@C=5mm{
\bullet\ar@{-}[drr]&\circ\ar@{-}[d]&\bullet\ar@{-}[dll]
\\
\bullet&\circ&\bullet}\quad\ \qquad
\xymatrix@R=10mm@C=5mm{
\bullet\ar@{-}[drr]&\bullet\ar@{-}[d]&\circ\ar@{-}[dll]
\\
\circ&\bullet&\bullet}\quad\ \qquad
\xymatrix@R=10mm@C=5mm{
\bullet\ar@{-}[drr]&\bullet\ar@{-}[d]&\bullet\ar@{-}[dll]
\\
\bullet&\bullet&\bullet}$$

These diagrams stand for the relations $abc=cba$, $abc^*=c^*ba$, and so on, up to $a^*b^*c^*=c^*b^*a^*$. For more about such pictures and relations, we refer to \cite{twe}.

There are some obvious equivalences between these relations, and by erasing the corresponding diagrams, we are led to three diagrams, namely:
$$
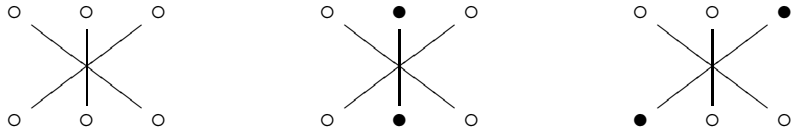
\xymatrix@R=10mm@C=5mm{
\circ\ar@{-}[drr]&\circ\ar@{-}[d]&\circ\ar@{-}[dll]
\\
\circ&\circ&\circ}\qquad\ \qquad
\xymatrix@R=10mm@C=5mm{
\circ\ar@{-}[drr]&\bullet\ar@{-}[d]&\circ\ar@{-}[dll]
\\
\circ&\bullet&\circ}\qquad\ \qquad
\xymatrix@R=10mm@C=5mm{
\circ\ar@{-}[drr]&\circ\ar@{-}[d]&\bullet\ar@{-}[dll]
\\
\bullet&\circ&\circ}$$

In order to extend now the $U_N^{**}$ geometry, the idea would be that of picking one of these diagrams, and using the corresponding relations. But here, we have:

\begin{proposition}
Consider the following types of relations, between abstract variables $a,b,c\in\{x_i\}$ subject to the relations $\sum_ix_ix_i^*=\sum_ix_i^*x_i=1$:
\begin{enumerate}
\item $abc=cba$.

\item $ab^*c=cb^*a$

\item $abc^*=c^*ba$.
\end{enumerate}
We have then $(1)\iff(3)\implies(2)$.
\end{proposition}

\begin{proof}
The equivalence $(1)\iff(3)$ follows from the following computations:
$$
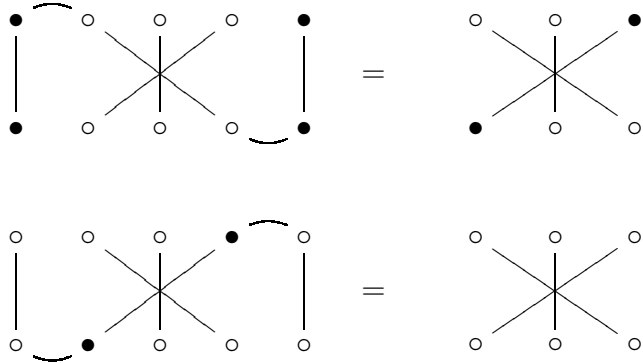
\xymatrix@R=10mm@C=5mm{
\bullet\ar@{-}[d]\ar@/^/@{-}[r]&\circ\ar@{-}[drr]&\circ\ar@{-}[d]&\circ\ar@{-}[dll]&\bullet\ar@{-}[d]\\
\bullet&\circ&\circ&\circ\ar@/_/@{-}[r]&\bullet}\quad
\xymatrix@R=4mm@C=6mm{&\\ =\\&\\& }
\xymatrix@R=10mm@C=6mm{
\circ\ar@{-}[drr]&\circ\ar@{-}[d]&\bullet\ar@{-}[dll]\\
\bullet&\circ&\circ}$$

$$\xymatrix@R=10mm@C=5mm{
\circ\ar@{-}[d]&\circ\ar@{-}[drr]&\circ\ar@{-}[d]&\bullet\ar@{-}[dll]\ar@/^/@{-}[r]&\circ\ar@{-}[d]\\
\circ\ar@/_/@{-}[r]&\bullet&\circ&\circ&\circ}\quad
\xymatrix@R=4mm@C=6mm{&\\ =\\&\\& }
\xymatrix@R=10mm@C=6mm{
\circ\ar@{-}[drr]&\circ\ar@{-}[d]&\circ\ar@{-}[dll]\\
\circ&\circ&\circ}$$

As for $(1+3)\implies(2)$, this is best worked out at the algebraic level, as follows:
$$ab^*c=\sum_dab^*cdd^*=\sum_dadcb^*d^*=\sum_dcdab^*d^*=\sum_dcb^*add^*=cb^*a$$

Thus we have indeed $(1)\iff(3)\implies(2)$, as claimed.
\end{proof}

Now by getting back to our problem, and more specifically, to our above-mentioned idea of using one diagram out of 3 possible ones, we can see, as a consequence of Proposition 3.1, that we have only one good choice, and are led to the following definition:

\begin{definition}
We have a sphere, a torus, and a quantum group, as follows:
\begin{enumerate}
\item $S^{N-1}_{\mathbb C,\times}\subset S^{N-1}_{\mathbb C,+}$, obtained via $ab^*c=cb^*a$, with $a,b,c\in\{x_i\}$.

\item $T_N^\times\subset\widehat{F_N}$, obtained via $ab^{-1}c=cb^{-1}a$, with $a,b,c\in\{g_i\}$.

\item $U_N^\times\subset U_N^+$, obtained via $ab^*c=cb^*a$, with $a,b,c\in\{u_{ij}\}$.
\end{enumerate}
\end{definition}

Observe that $U_N^\times$ is indeed a quantum group, containing $U_N^{**}$, and this because we are imposing to the standard coordinates of $U_N^+$ certain easy relations. It is of course possible to check Woronowicz's axioms in \cite{wo1} as well, directly. We will prove later on that we have indeed a noncommutative geometry, in the sense of Definition 1.5 above.

Let us clarify now a few more algebraic issues. The most elegant approach to the $U_N^\times$ geometry is in fact via projective space theory, using the following notions:

\begin{definition}
Given a subspace $X\subset S^{N-1}_{\mathbb C,+}$, we define quotients as follows,
\begin{enumerate}
\item Left projective version: $X\to PX$, with coordinates $p_{ij}=x_ix_j^*$,

\item Right projective version: $X\to P'X$, with coordinates $q_{ij}=x_j^*x_i$,

\item Full projective version: $X\to\mathcal PX$, with coordinates $p_{ij},q_{ij}$,
\end{enumerate}
and we say that $X$ is left/right/full half-classical when these spaces are classical.
\end{definition}

Observe that in the classical case, $X\subset S^{N-1}_\mathbb C$, the three projective versions coincide, and equal the usual projective version, obtained by dividing under the action of $\mathbb T$.

In the real case, $X\subset S^{N-1}_{\mathbb R,+}$, the three projective versions coincide as well, and $X$ is left or right half-classical when $X\subset S^{N-1}_{\mathbb R,*}$. This follows indeed from Theorem 2.8.

In relation now with $S^{N-1}_{\mathbb C,\times},S^{N-1}_{\mathbb C,**}$, we can use the following simple fact, from \cite{bbi}:

\begin{proposition}
Let $X\subset S^{N-1}_{\mathbb C,+}$, with coordinates $x_1,\ldots,x_N$.
\begin{enumerate}
\item $X\subset S^{N-1}_{\mathbb C,\times}$ precisely when $\{x_ix_j^*\}$ commute, and $\{x_i^*x_j\}$ commute as well.

\item $X\subset S^{N-1}_{\mathbb C,**}$ precisely when the variables $\{x_ix_j,x_ix_j^*,x_i^*x_j,x_i^*x_j^*\}$ all commute.
\end{enumerate}
\end{proposition}

\begin{proof}
Regarding the first assertion, the implication ``$\implies$'' follows from:
$$ab^*cd^*=cb^*ad^*=cd^*ab^*\quad,\quad a^*bc^*d=c^*ba^*d=c^*da^*b$$

As for the implication ``$\Longleftarrow$'', this is obtained as follows, by using the commutation assumptions in the statement, and by summing over $e=x_i$:
$$ae^*eb^*c=ab^*ce^*e=ce^*ab^*e=cb^*ee^*a\implies ab^*c=cb^*a$$

The proof of the second assertion is similar, because we can remove all the $*$ signs, except for those concerning $e^*$, and use the above computations with $a,b,c,d\in\{x_i,x_i^*\}$.
\end{proof}

With the above result in hand, we can now formulate:

\begin{proposition}
We have the following results:
\begin{enumerate}
\item $S^{N-1}_{\mathbb C,\times}$ is left and right half-classical, and is maximal with this property.

\item $S^{N-1}_{\mathbb C,**}$ is fully half-classical.

\item $S^{N-1}_{\mathbb R,*}$ is fully half-classical, and is maximal inside $S^{N-1}_{\mathbb R,+}$ with this property.
\end{enumerate}
\end{proposition}

\begin{proof}
All these assertions follow indeed from Proposition 3.4 above.
\end{proof}

We still have an issue to be clarified, namely that of proving that, in the diagram drawn in the beginning of this section, $U_N^\times$ sits indeed above $\mathbb TO_N^*$. But this comes from:

\begin{proposition}
We have $\mathbb TO_N^+\cap U_N^\times=\mathbb TO_N^*$, as quantum subgroups of $U_N^+$.
\end{proposition}

\begin{proof}
According to the definition of $\mathbb TO_N^*$, from section 1, this quantum group appears as $\mathbb TO_N^*=\mathbb TO_N^+\cap U_N^{**}$. Thus, we must  prove that we have $\mathbb TO_N^+\cap U_N^\times\subset U_N^{**}$.

In terms of defining relations, we must prove that, from $ab^*=a^*b$ and $ab^*c=cb^*a$ for any $a,b,c\in\{u_{ij}\}$, we can deduce that we have $abc=cba$, for any $a,b,c\in\{u_{ij},u_{ij}^*\}$.

But this is clear, because by using $ab^*=a^*b$, we can first obtain $a^*bc=cba^*$, and then, by using Proposition 3.1, we can obtain from this the other relations as well.

Here we have used the fact that what we know about abstract variables satisfying $\sum_ix_ix_i^*=\sum_ix_i^*x_i=1$ applies to the coordinates to any closed subgroup $G\subset U_N^+$, simply because these coordinates, when rescaled by $\sqrt{N}$, do satisfy these relations.
\end{proof}

We recall that the free complexification of a compact quantum group $G$, with standard coordinates denoted $v_{ij}$, is the compact quantum group $\widetilde{G}$ corresponding to the subalgebra $C(\widetilde{G})\subset C(\mathbb T)*C(G)$ generated by the variables $u_{ij}=zv_{ij}$, where $z$ is the standard generator of $C(\mathbb T)$. Observe that $\widetilde{G}$ is indeed a quantum group, because it appears as a subgroup of $\mathbb T\,\hat{*}\,G$, the quantum group associated to $C(\mathbb T)*C(G)$. See \cite{rau}.

Following \cite{ba3}, we have the following result:

\begin{proposition}
The quantum groups $O_N^*,U_N,U_N^{**},U_N^\times$ have the following properties:
\begin{enumerate}
\item They have the same left projective version, equal to $PU_N$.

\item They have the same free complexification, equal to $U_N^\times$.
\end{enumerate} 
\end{proposition}

\begin{proof}
With terminology and notations from \cite{ba3}, the idea is as follows:

(1) Here $PO_N^*=PU_N$ is a well-known result, from \cite{bve}. It is clear as well that we have $PU_N\subset PU_N^{**}\subset PU_N^\times$. Now by using Proposition 3.4 (1), we conclude that $PU_N^\times$ is classical, and so we must have $PU_N^\times\subset(PU_N^+)_{class}$. But this latter space $(PU_N^+)_{class}$ is known to be equal to $PU_N$, and this finishes the proof. See \cite{ba3}, \cite{bdd}, \cite{bd+}.

(2) If we denote by $v_{ij}$ the standard coordinates on $G=O_N^*,U_N,U_N^{**},U_N^\times$, and by $z$ the generator of a copy of $C(\mathbb T)$, free from $C(G)$, then with $a,b,c\in\{v_{ij}\}$ we have:
$$(za)(zb)^*(zc)=zab^*c=zcb^*a=(zc)(zb)^*(za)$$

Thus we have $\widetilde{G}\subset U_N^\times$. Conversely now, it follows from the general theory of the free complexifications of easy quantum groups \cite{rau} that both $K=\widetilde{G},U_N^\times$ should appear as free complexifications of certain intermediate easy quantum groups $O_N\subset H\subset O_N^+$. On the other hand, since we have $PH=P\widetilde{H}=PK=PU_N$, the only choice here is $H=O_N^*$. Thus we have $\widetilde{G}=U_N^\times=\widetilde{O_N^*}$, and this finishes the proof. See \cite{ba3}.
\end{proof}

Let us prove now the quantum isometry group result. This is known since \cite{ba3}, but we have now a much simpler proof for this fact. We use the following notion:

\begin{definition}
Given a closed subgroup $G\subset U_N^+$, and a closed subset $X\subset S^{N-1}_{\mathbb C,+}$, we say that $G$ acts projectively on $X$, and write $PG\curvearrowright PX$, when the formula $$\Phi(x_ix_j^*)=\sum_{ab}u_{ia}u_{jb}^*\otimes x_ax_b^*$$
defines a morphism of $C^*$-algebras $\Phi:C(PX)\to C(PG)\otimes C(PX)$.
\end{definition}

Observe the similarity with Definition 1.3 above, dealing with the affine case. As in the affine case, such a morphism is automatically coassociative and counital. Observe also that any affine action $G\curvearrowright X$ produces a projective action $PG\curvearrowright PX$. See \cite{bme}.

We can now formulate our quantum isometry group results, as follows:

\begin{proposition}
We have the following results:
\begin{enumerate}
\item $PG\curvearrowright P^{N-1}_\mathbb C$ implies $G\subset U_N^\times$.

\item $G\curvearrowright S^{N-1}_{\mathbb C,\times}$ implies $G\subset U_N^\times$.
\end{enumerate}
\end{proposition}

\begin{proof}
Since $G\curvearrowright S^{N-1}_{\mathbb C,\times}$ implies that we have $PG\curvearrowright PS^{N-1}_{\mathbb C,\times}=P^{N-1}_\mathbb C$, we just have to prove the first assertion. For this purpose, we use an old method from \cite{bhg}, as in \cite{ba4}.

Consider indeed a coaction map, written $\Phi(p_{ij})=\sum_{ab}u_{ia}u_{jb}^*\otimes p_{ab}$, with $p_{ab}=z_a\bar{z}_b$. The idea will be that of using the formula $p_{ab}p_{cd}=p_{ad}p_{cb}$. We have:
\begin{eqnarray*}
\Phi(p_{ij}p_{kl})&=&\sum_{abcd}u_{ia}u_{jb}^*u_{kc}u_{ld}^*\otimes p_{ab}p_{cd}\\
\Phi(p_{il}p_{kj})&=&\sum_{abcd}u_{ia}u_{ld}^*u_{kc}u_{jb}^*\otimes p_{ad}p_{cb}
\end{eqnarray*}

The left terms being equal, and the last terms on the right being equal too, we deduce that, with $[a,b,c]=abc-cba$, we must have the following equality:
$$\sum_{abcd}u_{ia}[u_{jb}^*,u_{kc},u_{ld}^*]\otimes p_{ab}p_{cd}=0$$

In order to exploit this equality, we use basic tensor product theory, the trick when having a formula of type $\sum_ia_i\otimes b_i=0$ being to compact the elements on the right, by using linear dependence, then to conclude that the elements on the left must vanish.

In our case, since the quantities $p_{ab}p_{cd}=z_a\bar{z}_bz_c\bar{z}_d$ on the right depend only on the numbers $|\{a,c\}|,|\{b,d\}|\in\{1,2\}$, and this dependence produces the only possible linear relations between the variables $p_{ab}p_{cd}$, we are led to $2\times2=4$ equations, as follows:

(1) $u_{ia}[u_{jb}^*,u_{ka},u_{lb}^*]=0$, $\forall a,b$.

(2) $u_{ia}[u_{jb}^*,u_{ka},u_{ld}^*]+u_{ia}[u_{jd}^*,u_{ka},u_{lb}^*]=0$, $\forall a$, $\forall b\neq d$.

(3) $u_{ia}[u_{jb}^*,u_{kc},u_{lb}^*]+u_{ic}[u_{jb}^*,u_{ka},u_{lb}^*]=0$, $\forall a\neq c$, $\forall b$.

(4) $u_{ia}([u_{jb}^*,u_{kc},u_{ld}^*]+[u_{jd}^*,u_{kc},u_{lb}^*])+u_{ic}([u_{jb}^*,u_{ka},u_{ld}^*]+[u_{jd}^*,u_{ka},u_{lb}^*])=0,\forall a\neq c,\forall b\neq d$.

Let us process now all these formulae. Regarding (3,4), we won't need them, in what follows. From (1,2) we conclude that (2) holds with no restriction on the indices. By multiplying now this formula to the left by $u_{ia}^*$, and then summing over $i$, we obtain:
$$[u_{jb}^*,u_{ka},u_{ld}^*]+[u_{jd}^*,u_{ka},u_{lb}^*]=0$$

By applying now the antipode, then the involution, and finally by suitably relabelling all the indices, we successively obtain from this formula:
\begin{eqnarray*}
&&[u_{dl},u_{ak}^*,u_{bj}]+[u_{bl},u_{ak}^*,u_{dj}]=0\\
&\implies&[u_{dl}^*,u_{ak},u_{bj}^*]+[u_{bl}^*,u_{ak},u_{dj}^*]=0\\
&\implies&[u_{ld}^*,u_{ka},u_{jb}^*]+[u_{jd}^*,u_{ka},u_{lb}^*]=0
\end{eqnarray*}

Now by comparing with the original relation, above, we conclude that we have:
$$[u_{jb}^*,u_{ka},u_{ld}^*]=[u_{jd}^*,u_{ka},u_{lb}^*]=0$$

Thus we have reached to the formulae defining $U_N^\times$, and we are done.
\end{proof}

We have now all the needed ingredients for proving:

\begin{theorem}
The sphere $S^{N-1}_{\mathbb R,\times}$, the torus $T_N^\times$ and the quantum group $U_N^\times$ form a noncommutative geometry. This geometry is easy, and not amenable.
\end{theorem}

\begin{proof}
The verification of all the axioms is standard, with the only non-trivial fact, namely the universality of the action $U_N^\times\curvearrowright S^{N-1}_{\mathbb C,\times}$, coming from Proposition 3.9 above.

Regarding the non-amenability claim for the geometry that we have, here we must prove that the quantum group $U_N^\times$ is not coamenable. But this can  be checked by using a free complexification trick. We know that the discrete group $\Gamma_N^\times=\widehat{T_N^\times}$ is given by:
$$\Gamma_N^\times=\left<g_1,\ldots,g_N\Big|ab^{-1}c=cb^{-1}a,\forall a,b,c\in\{g_i\}\right>$$

Now observe that, if we denote by $h_1,\ldots,h_N$ the standard generators of $\mathbb Z^N$, and by $z$ the generator of a copy of $\mathbb Z$, which is free from $\mathbb Z^N$, then with $a,b,c\in\{h_i\}$ we have:
$$(za)(zb)^{-1}(zc)=zab^{-1}c=zcb^{-1}a=(zc)(zb)^{-1}(za)$$

We therefore have a group morphism $\Gamma_N^\times\to\mathbb Z*\mathbb Z^N$, given by $g_i\to zh_i$. Now observe that the image of this morphism contains the following two elements:
$$(zh_1)^{-1}(zh_2)=h_1^{-1}h_2\quad,\quad 
(zh_1)(zh_2)^{-1}=zh_1h_2^{-1}z^*$$

These elements being free, we obtain a copy of $F_2$ inside the image. Thus the image, and then $\Gamma_N^\times$, and then $\widehat{U_N^\times}$ itself, follow to be non-amenable as well.
\end{proof}

The $U_N^\times$ geometry can be further investigated by using various algebraic tricks, and notably the free complexification formula $U_N^\times=\widetilde{U_N}$. We refer here to \cite{ba3}.

\section{Half-classical geometry}

In this section we introduce and study the $U_N^*$ geometry, which is the ``correct'' complex half-classical one, in the sense that it is the biggest half-classical geometry. 

In view of Definition 3.3 above, we can indeed formulate:

\begin{definition}
We have a sphere, a torus, and a quantum group, as follows:
\begin{enumerate}
\item $S^{N-1}_{\mathbb C,*}\subset S^{N-1}_{\mathbb C,+}$, obtained via ``$ab^*,a^*b$ all commute'', with $a,b,c\in\{x_i\}$.

\item $T_N^*\subset\widehat{F_N}$, obtained via ``$ab^{-1},a^{-1}b$ all commute'', with $a,b,c\in\{g_i\}$.

\item $U_N^*\subset U_N^+$, obtained via ``$ab^*,a^*b$ all commute'', with $a,b,c\in\{u_{ij}\}$.
\end{enumerate}
In other words, these are the biggest half-classical sphere, torus, and quantum group.
\end{definition}

As a first remark, the real version of $S^{N-1}_{\mathbb C,*}$, obtained by imposing the conditions $x_i=x_i^*$ to the standard coordinates, is the half-classical real sphere $S^{N-1}_{\mathbb R,*}$. Observe also that we have inclusions as follows, coming from the various results in Proposition 3.4:
$$S^{N-1}_{\mathbb C,**}\subset S^{N-1}_{\mathbb C,*}\subset S^{N-1}_{\mathbb C,\times}$$

Similar inclusions are valid for the tori, and for the quantum groups. Finally, observe that $U_N^*$ is by definition easy, coming from the following two diagrams:
$$
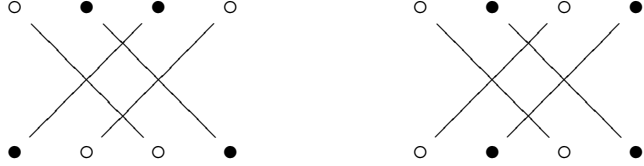
\xymatrix@R=15mm@C=5mm{\circ\ar@{-}[drr]&\bullet\ar@{-}[drr]&\bullet\ar@{-}[dll]&\circ\ar@{-}[dll]\\\bullet&\circ&\circ&\bullet}\qquad \quad\qquad 
\xymatrix@R=15mm@C=5mm{\circ\ar@{-}[drr]&\bullet\ar@{-}[drr]&\circ\ar@{-}[dll]&\bullet\ar@{-}[dll]\\\circ&\bullet&\circ&\bullet}$$

In order to better understand the construction $X\to\mathcal PX$, and the definition of $S^{N-1}_{\mathbb C,*}$ itself, let us perform now some explicit computations. We denote by $P^{N-1}_\mathbb C$ the usual complex projective space. We have the following result:

\begin{proposition}
The projective versions of $S^{N-1}_{\mathbb C,**}\subset S^{N-1}_{\mathbb C,*}\subset S^{N-1}_{\mathbb C,\times}$ are given by
$$\xymatrix@R=20mm@C=20mm{
S^{N-1}_{\mathbb C,**}\ar[r]\ar[d]&S^{N-1}_{\mathbb C,*}\ar[r]\ar[d]&S^{N-1}_{\mathbb C,\times}\ar[d]\\
P^{N-1}_\mathbb C\times P^{N-1}_\mathbb C\ar[r]&P^{N-1}_\mathbb C\times P^{N-1}_\mathbb C\ar[r]&P^{N-1}_\mathbb C\circ P^{N-1}_\mathbb C
}$$
where the product on the bottom right is constructed by conjugating by a free unitary.
\end{proposition}

\begin{proof}
We use the following presentation result, which comes from the Gelfand theorem, and from the fact that $P^{N-1}_\mathbb C$ is the space of rank 1 projections in $M_N(\mathbb C)$:
$$C(P^{N-1}_\mathbb C)=C^*_{comm}\left\{(p_{ij})_{i,j=1,\ldots,N}\Big|p=p^2=p^*,Tr(p)=1\right\}$$

Let us first discuss the computation of the spaces $\mathcal PS^{N-1}_{\mathbb C,**}\subset\mathcal PS^{N-1}_{\mathbb C,*}$. We know that these spaces are both classical. We also know that the left and right components, in the sense of Definition 3.3 above, of these spaces are all equal to $P^{N-1}_\mathbb C$, for instance because their standard generators satisfy the above defining relations for $C(P^{N-1}_\mathbb C)$.

In order to finish, it remains to prove that the subspaces $PS^{N-1}_{\mathbb C,**},P'S^{N-1}_{\mathbb C,**}\subset\mathcal PS^{N-1}_{\mathbb C,**}$, which are both isomorphic to $P^{N-1}_\mathbb C$, are in generic position. For this purpose, we can use a suitable matrix model, coming from \cite{bbi}. Let indeed $u_i,v_i$ be the standard coordinates of two independent copies of $S^{N-1}_\mathbb C$, and consider the following matrices:
$$X_i=\begin{pmatrix}0&u_i\\ v_i&0\end{pmatrix}\quad,\quad X_i^*=\begin{pmatrix}0&\bar{v}_i\\ \bar{u}_i&0\end{pmatrix}$$

We have then $\sum_iX_iX_i^*=\sum_iX_i^*X_i=1$ and the relations $abc=cba$ hold as well, for any $a,b,c\in\{X_i,X_i^*\}$. Thus we have a matrix model, as follows:
$$C(S^{N-1}_{\mathbb C,**})\to M_2(C(S^{N-1}_\mathbb C\times S^{N-1}_\mathbb C))\quad,\quad x_i\to X_i$$

The point now is that, in this model, we have the following formulae:
$$X_iX_j^*=\begin{pmatrix}u_i\bar{u}_j&0\\0&v_i\bar{v}_j\end{pmatrix}\quad,\quad
X_j^*X_i=\begin{pmatrix}v_i\bar{v}_j&0\\0&u_i\bar{u}_j\end{pmatrix}$$

Now since these matrices are conjugated by an order 2 automorphism, the algebra that they generate is isomorphic to $C(P^{N-1}_\mathbb C\times P^{N-1}_\mathbb C)$, and this finishes the proof. 

Finally, regarding the computation for $S^{N-1}_{\mathbb C,\times}$, let us denote by $p_{ij}=z_i\bar{z}_j$ the standard coordinates on $P^{N-1}_\mathbb C=PS^{N-1}_\mathbb C$. In the usual free complexification model for $S^{N-1}_{\mathbb C,\times}$, namely $\widetilde{S}^{N-1}_\mathbb C$, we have then $x_ix_j^*=p_{ij},x_j^*x_i=zp_{ij}z^*$, and this gives the result.
\end{proof}

In order to verify now the axioms, we follow the proof for $U_N^\times$. First, we have:

\begin{definition}
Given a closed subgroup $G\subset U_N^+$, and a closed subset $X\subset S^{N-1}_{\mathbb C,+}$, we say that $G$ acts fully projectively on $X$, and write $\mathcal PG\curvearrowright\mathcal PX$, when the formulae 
$$\Phi(x_ix_j^*)=\sum_{ab}u_{ia}u_{jb}^*\otimes x_ax_b^*$$
$$\Phi(x_i^*x_j)=\sum_{ab}u_{ia}^*u_{jb}\otimes x_a^*x_b$$
define a morphism of $C^*$-algebras $\Phi:C(\mathcal PX)\to C(\mathcal PG)\otimes C(\mathcal PX)$.
\end{definition}

As in the affine case, such a morphism is automatically coassociative and counital. Observe also that any affine action $G\curvearrowright X$ produces a projective action $\mathcal PG\curvearrowright\mathcal PX$.

We can now formulate our quantum isometry group results, as follows:

\begin{proposition}
We have the following results:
\begin{enumerate}
\item We have an affine action $U_N^*\curvearrowright S^{N-1}_{\mathbb C,*}$.

\item $\mathcal PG\curvearrowright P^{N-1}_\mathbb C\times P^{N-1}_\mathbb C$ implies $G\subset U_N^*$.

\item $G\curvearrowright S^{N-1}_{\mathbb C,*}$ implies $G\subset U_N^*$.

\item $U_N^*$ is the quantum isometry group of $S^{N-1}_{\mathbb C,*}$.
\end{enumerate}
\end{proposition}

\begin{proof}
Our first claim is that it is enough to prove (2). Indeed, in order to prove (1), observe that with $X_i=\sum_au_{ia}\otimes x_a$, we have the following formulae:
$$X_iX_j^*=\sum_{ab}u_{ia}u_{jb}^*\otimes x_ax_b^*\quad,\quad X_j^*X_i=\sum_{ab}u_{jb}^*u_{ia}\otimes x_b^*x_a$$

Now since the various variables on the right pairwise commute, the variables on the left commute as well, and so we can define the action map, by $x_i\to X_i$.

The other remark is that since $G\curvearrowright S^{N-1}_{\mathbb C,*}$ implies $\mathcal PG\curvearrowright PS^{N-1}_{\mathbb C,\times}=P^{N-1}_\mathbb C\times P^{N-1}_\mathbb C$, we have $(2)\implies(3)$. Finally, the implication $(1+3)\implies(4)$ is trivial.

In order to prove now (2), observe that $\mathcal PG\curvearrowright P^{N-1}_\mathbb C\times P^{N-1}_\mathbb C$ implies $PG\curvearrowright P^{N-1}_\mathbb C$. Thus, we can use Proposition 3.9 (1), and we obtain $G\subset U_N^\times$. 

We are therefore left with proving that $\mathcal PG\curvearrowright P^{N-1}_\mathbb C\times P^{N-1}_\mathbb C$ implies that the following diagram belongs to the Tannakian category of $G$:
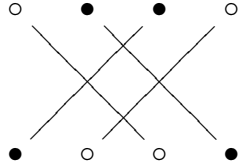
$$\xymatrix@R=15mm@C=5mm{\circ\ar@{-}[drr]&\bullet\ar@{-}[drr]&\bullet\ar@{-}[dll]&\circ\ar@{-}[dll]\\\bullet&\circ&\circ&\bullet}$$

For this purpose, consider indeed a coaction map, written as in Definition 4.3. By multiplying the two relations there, we obtain:
$$\Phi(x_ix_j^*x_k^*x_l)=\sum_{abcd}u_{ia}u_{jb}^*u_{kc}^*u_{ld}\otimes x_ax_b^*x_c^*x_d$$
$$\Phi(x_k^*x_lx_ix_j^*)=\sum_{abcd}u_{kc}^*u_{ld}u_{ia}u_{jb}^*\otimes x_c^*x_dx_ax_b^*$$

Assuming now that $x_1,\ldots,x_N$ are the standard coordinates on $S^{N-1}_{\mathbb C,*}$, the products of $x$ variables at left are equal, and so are the products at right. Thus, we have:
$$\sum_{abcd}[u_{ia}u_{jb}^*,u_{kc}^*u_{ld}]\otimes x_ax_b^*x_c^*x_d=0$$

Now recall that, in view of Proposition 4.2 above, we can write $x_ax_b^*x_c^*x_d=p_{ab}\otimes q_{dc}$, where $p_{ab},q_{cd}$ are the standard coordinates on $P^{N-1}_\mathbb C$. Thus, our formula becomes:
$$\sum_{abcd}[u_{ia}u_{jb}^*,u_{kc}^*u_{ld}]\otimes p_{ab}\otimes q_{dc}=0$$

Now since the variables on the right are linearly independent, we obtain that all the commutators vanish, and this finishes the proof.
\end{proof}

We can now formulate our main result, as follows:

\begin{theorem}
We have a noncommutative geometry, with sphere $S^{N-1}_{\mathbb C,*}$, torus $T_N^*$ and quantum group $U_N^*$. This is the biggest geometry having a classical projective version.
\end{theorem}

\begin{proof}
The verification of all the axioms is standard, with the only non-trivial fact, namely the universality of the action $U_N^*\curvearrowright S^{N-1}_{\mathbb C,*}$, coming from Proposition 4.4 above. Regarding the last assertion, where ``biggest'' means as usual ``maximal'', this is clear.
\end{proof}

In view of the above result, the $U_N^*$ geometry as constructed above seems to be the ``correct'' complex half-classical geometry. This geometry waits of course to be developed, with the potential questions concerning the submanifolds $X\subset S^{N-1}_{\mathbb C,*}$ being a priori as many as the questions concerning the submanifolds $X\subset S^{N-1}_\mathbb C$, or perhaps $X\subset S^{N-1}_\mathbb R$, with the remark of course that there are whole books written on these latter manifolds. As a very first question here, interesting would be to work out the analogues of \cite{bic}, \cite{bdu}, in the present setting. Finally, we conjecture that the $U_N^*$ geometry is amenable, and is moreover maximal with this amenability property, at least in the easy framework.

\end{document}